\documentclass{amsart}
\usepackage{amssymb,latexsym}
\usepackage{amsmath}
\usepackage{graphicx}
\usepackage{xcolor}
\usepackage{enumerate}
\usepackage{newlattice}
\usepackage{gensymb}
\usepackage{xspace}

\newcommand{\swing}{\mathbin{\raisebox{2.0pt}
       {\rotatebox{160}{$\curvearrowleft$}}}}

\newcommand{\grrel}{\mathbin{\gr}}

\theoremstyle{plain}
 \newtheorem{theorem}{Theorem}
 \newtheorem{lemma}[theorem]{Lemma}
 \newtheorem{corollary}[theorem]{Corollary}
 
 \theoremstyle{definition}
 \newtheorem{definition}[theorem]{Definition}

\newcommand{\pr}[1]{\tup{(}#1\tup{)}}
\newcommand{\Col}[1]{\tup{col}(#1)}
\newcommand{\col}[1]{\tup{col}(#1)}

\newcommand{\traj}[1]{\tup{traj}{(#1})}
\newcommand{\ZL}[1]{\ell(#1)}

\newcommand{\normup}{\includegraphics[scale=.2]{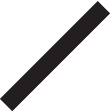}-}
\newcommand{\normdn}{\includegraphics[scale=.2]{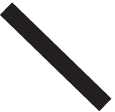}-}
\newcommand{\steep}{\includegraphics[scale=.2]{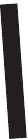}-}

\newcommand{\normaldown}{\includegraphics[scale=.2]{normaldn}-}
\newcommand{\normn}{\includegraphics[scale=.2]{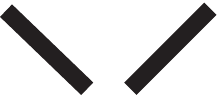}}
\newcommand{\normupn}{\includegraphics[scale=.2]{normalup}}
\newcommand{\normdnn}{\includegraphics[scale=.2]{normaldn}}
\newcommand{\steepn}{\includegraphics[scale=.2]{steep}}

\newcommand{\normalupn}{\includegraphics[scale=.2]{normalup}}

\begin{document}
\title[The Swing Lemma and Cz\'edli diagrams]
{Using the Swing Lemma and Cz\'edli diagrams 
for congruences of planar semimodular lattices}
\author[G.\ Gr\"atzer]{George Gr\"atzer}
\email{gratzer@me.com}
\urladdr{http://server.maths.umanitoba.ca/homepages/gratzer/}
\address{University of Manitoba}
\date{April 12, Version 0.4}
\begin{abstract} 
A planar semimodular lattice $K$ is \emph{slim} 
if $\SM{3}$ is not a sublattice of~$K$.
In a recent paper, G. Cz\'edli found four new properties  
of congruence lattices of slim, planar, semimodular lattices,
including the \emph{No Child Property}: 
\emph{Let~$P$ be the ordered set of join-irreducible congruences of $K$.
Let $x,y,z \in  P$ and let $z$ be a~maximal element of $P$.
If  $x \neq y$, $x, y \prec z$ in $P$, 
then there is no element $u$ of $P$ such that $u \prec x, y$ in $P$.}

We are applying my Swing Lemma, 2015,
and a type of standardized diagrams of Cz\'edli's, to verify Cz\'edli's four properties.
\end{abstract}



\maketitle    

\section{Introduction}\label{S:Introduction}

Let $K$ be a planar semimodular lattice. 
We call the lattice $K$ \emph{slim} if $\SM{3}$ is not a~sublattice of~$K$.
In the paper \cite{gG14a}, I found a property of congruences of  slim, planar, semimodular lattices.
In the same paper (see also Problem 24.1 in G. Gr\"atzer~\cite{CFL2}),
I~proposed the following:

\vspace{6pt}

\tbf{Problem 1.} Characterize the congruence lattices of slim  planar semimodular lattices.

\vspace{6pt}

G. Cz\'edli ~\cite[Corollaries 3.4, 3.5, Theorem 4.3]{gCa} found four new properties  
of congruence lattices of slim, planar, semimodular lattices.

\begin{named}{Cz\'edli's Theorem~\cite{gCa}}
Let $K$ be a slim, planar, semimodular  lattice 
with at least three elements and let~$\E P$ be
the ordered set of join-irreducible congruences of $K$.

\begin{enumeratei}
\item \emph{Partition Property:} 
The set of maximal elements of $\E P$ can be represented as the disjoint union 
of two nonempty subsets such that no two distinct elements 
in the same subset have a common lower cover.\label{E:LC} 
\item \emph{Maximal Cover Property:}  
If $x \in \E P$ is covered by a maximal element $y$ of $\E P$, 
then $y$ is not the only cover.
\item \emph{Four-Crown Two-pendant Property:}
There is no cover-preserving embedding of the ordered set $\E R$ in Figure~\ref{F:notation}
into $\E P$ satisfying the property\tup{:} any maximal element of~$\E R$
is a maximal element of $\E P$.
\item \emph{No Child Property:} 
Let $x \neq y \in \E P$ and let $z$ be a maximal element of $\E P$. 
Let us assume that both $x$ and $y$ are covered by $z$ in $\E P$. 
Then there is no element $u \in \E P$ such that $u$ is covered by $x$ and $y$.
\end{enumeratei}
\end{named}

In this paper, we will prove this theorem using the Swing Lemma 
and Cz\'edli diagrams, see Sections~\ref{S:Swing} and \ref{S:diagrams}.
By G. Gr\"atzer and E. Knapp \cite{GKn09}, 
every slim, planar, semimodular lattice $K$ has a congruence-preserving
extension $\ol K$ to a slim rectangular lattice. 
Any of the properties (i)--(iv) holds for $K$ if{}f it holds for $\ol K$. 
Therefore, in the rest of this paper, we can assume that $K$ is  a slim rectangular lattice,
simplifying the discussion.

\begin{figure}[t!]
\centerline
{\includegraphics[scale=1]{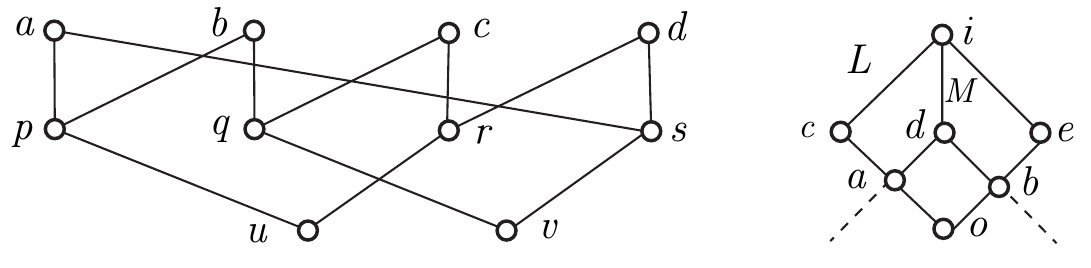}}
\caption{The Four-crown Two-pendant ordered set $\E R$ with notation; 
the covering $\SN 7$ sublattice with notation}\label{F:notation}
\end{figure}

\subsection*{Outline} Section~\ref{S:Background} provides the results we need: 
the Swing Lemma, Cz\'edli's  $\E C_1$-diagrams (we call them Cz\'edli diagrams),
and forks.
Section~\ref{S:partition} proves the Partition Property, 
Section~\ref{S:Maximal} does the Maximal Cover Property,
while Section~\ref{S:Child} verifies the No Child Property.
Finally,  The Four-Crown Two-pendant Property is proved in Section~\ref{S:Crown}.

\section{Background}\label{S:Background}

Most basic concepts and notation not defined in this paper 
are available in Part~I of the book \cite{CFL2}, see 

\verb+https://www.researchgate.net/publication/299594715+\\
\indent \verb+arXiv:2104.06539+

\noindent It is available to the reader. 
We will reference it, for instance, as [CFL2, page 52].

\subsection{Swing Lemma}\label{S:Swing}

An  \emph{SPS lattice} $K$ is a slim, planar, and semimodular lattice, see [CFL2, Chapter 4]. 
For an edge (prime interval) $E$ of $K$, let $E = [0_E, 1_E]$ 
and define  $\Col{E}$,  the \emph{color of}~$E$, as $\con E$,
the (join-irreducible) congruence generated by collapsing $E$ (see [CFL2, Section 3.2]).
We write $\E P$ for $\Ji {\Con K}$, the ordered set of join-irreducible congruences of $K$.

As in my paper~\cite{gG15},  for the edges $U, V$ 
of an SPS lattice $K$, we define a binary relation:
$U$~\emph{swings} to $V$, in formula, $U \swing V$, if $1_U = 1_V$, 
the element $1_U = 1_V$ covers at least three elements,
and $0_V$ is neither the left-most nor the right-most element covered by $1_U = 1_V$; 
if also $0_U$ is such, then the swing is \emph{interior}.

\begin{named}{Swing Lemma [G. Gr\"atzer~\cite{gG15}]}
Let $K$ be an SPS lattice and let $U$ and $V$ be edges in $K$. 
Then  $\Col V \leq\Col U$ if{}f there exists an edge $R$ 
such that $U$ is up-perspective to $R$ 
and there exists a sequence of edges and a~sequence of binary relations 
\begin{equation}\label{E:sequence}
   R = R_0 \grrel_1 R_1 \grrel_2 \dots \grrel_n R_n = V,
\end{equation}
where each relation $\grrel_i$ is $\perspdn$ \pr{down-perspective} or $\swing$ \pr{swing}.

In~addition, the sequence \eqref{E:sequence} also satisfies 
\begin{equation}\label{E:geq}
   1_{R_0} \geq 1_{R_1} \geq \dots \geq 1_{R_n}.
\end{equation}
\end{named}

The following statements are immediate consequences of the Swing Lemma,
see my papers~\cite{gG15} and \cite{gG14e}.

\begin{corollary}\label{C:ucovv} We use the assumptions of the Swing Lemma.
\begin{enumeratei}
\item If $\Col U = \Col V$, then either $U$ and $V$ are perspective
or there exist edges $S$ and $T$ 
so that $U \perspup S \swing T \perspdn V$, where the swing is interior.
\item If $v \prec u$ in $\E P$, then there exist edges $U, L$  of color $u$
and $R, V$ of color~$v$ such that~$U \perspup L \swing R \perspdn V$,
as in see the first diagram of Figure~\ref{F:Corollary2}.
\end{enumeratei}
\end{corollary}
Note that the covering relation in $\E P$ can always be represented in a covering $\SN 7$.

\begin{figure}[htb]
\centerline{\includegraphics{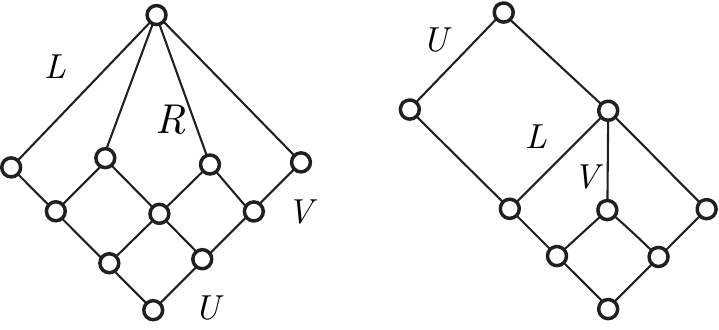}}
\caption{Illustrating Corollaries 2 and 3}
\label{F:Corollary2}
\end{figure}

\begin{corollary}\label{C:max}
Let the edge $U$ be on the upper edge of $K$. 
Then $\Col U$ is a maximal element of $\E P$.
Conversely, if $u$ is a maximal element of $\E P$,
then there is an edge $U$ on the upper edge of $K$ so that $\Col U = u$.
\end{corollary}

\begin{corollary}\label{C:ucovv1}
Let $v \prec u$ in  $\E P$ and $u$ be a maximal element of $\E P$.
If $U$ is an edge on the upper edge of $K$ with $\Col U = u$,
then there exist an edge $L$ 
such that $U \perspdn L \swing V$,
as in the second diagram of Figure~\ref{F:Corollary2}.
\end{corollary}

\subsection{Cz\'edli diagrams}\label{S:diagrams}

In the diagram of a planar lattice $K$,
a \emph{normal edge} (\emph{line}) has a slope of $45\degree$ or $135\degree$.
If it is the first, we call it a \emph{normal-up edge} (\emph{line}), 
otherwise, a  \emph{normal-down edge} (\emph{line}).
Any edge of slope strictly between $45\degree$ and $135\degree$ is \emph{steep}.
We use the symbols \normn, \normupn, \normdnn, and \steepn, 
for normal, normal-up, normal-down, and steep, respectively.

\begin{definition}[G. Cz\'edli \cite{gC1}]\label{D:well}
A diagram of an SPS lattice $L$ is a
\emph{Cz\'edli-diagram} if the middle edge of any covering $\SN 7$ is steep
and all other edges are normal.
\end{definition}

\begin{theorem}[G. Cz\'edli \cite{gC1}]\label{T:well}
Every slim, planar, semimodular lattice $K$ has a Cz\'edli diagram.
\end{theorem}

See the illustrations in this paper for examples of Cz\'edli diagrams.
G. Cz\'edli \cite{gC1} calls these $\E C_1$-diagrams. 
He also defines  $\E C_0$- and  $\E C_2$-diagrams. 
For an alternative proof for the existence of Cz\'edli diagrams, see G.~Gr\"atzer~\cite{gG21b}.

\emph{In this paper, $K$ denotes a slim rectangular lattice with a fixed Cz\'edli diagram}.

Let $C$ and $D$ be maximal chains in an interval $[a,b]$ of $K$ 
such that $C \ii D = \set{a,b}$.
If there is no element of~$K$ between $C$ and $D$, 
then we call $C \uu D$ a~\emph{cell}. 
A~four-element cell is a \text{\emph{$4$-cell}}. 
Opposite edges of a $4$-cell are called \emph{adjacent}.
Planar semimodular lattices are $4$-cell lattices, 
that is, all of its cells are $4$-cells,
see G.~Gr\"atzer and E. Knapp \cite[Lemmas 4, 5]{GKn07} 
and  [CFL2,~Section 4.1] for more detail.

The following statement illustrates the use of Cz\'edli diagrams.

\begin{lemma}\label{L:application}
Let $K$ be a slim rectangular lattice $K$ with a fixed Cz\'edli diagram
and let $X$ be a \normup edge of $K$. 
Then $X$ is up-perspective either to an edge in the upper-left boundary of $K$
or to a \steep edge.
\end{lemma}

\begin{proof}
If $X$ is not in the upper-left boundary of $K$, 
then there is a $4$-cell $C$ whose lower-right edge is $X$.
If the upper-left edge is \steepn\ or it is in the upper-left boundary, then we are done. 
Otherwise, we proceed the same way until we reach the upper-left boundary.
\end{proof}

\subsection{Trajectories}\label{S:Trajectories}
G. Cz\'edli and E.\,T. Schmidt \cite{CS13} introduced a \emph{trajectory} in $K$ 
as a maximal sequence of consecutive edges, see also [CFL2, Section~4.1]. 
The \emph{top edge}~$T$ of a trajectory 
is either in the upper boundary of $K$ or it is \steepn. 
For such an edge~$T$, we denote by $\traj T$ the trajectory with top edge $T$.
Since an element $a$ in a slim semimodular lattice 
has at most three covers (G.~Gr\"atzer and E. Knapp \cite[Lemma 8]{GKn07}),
a trajectory has at most one top edge and at most one \steep-edge.
So we conclude:

\begin{lemma}\label{L:disj}
Let $K$ be a slim rectangular lattice $K$ with a fixed Cz\'edli diagram.
Let $X$ and $Y$ be \steep-edges of $K$. 
Then $\traj X$ and $\traj Y$ are disjoint.
\end{lemma}

%
%

\section{The Partition Property}\label{S:partition}
We start with a lemma.

\begin{lemma}\label{L:disjoint}
Let $X $ and $Y$ be distinct edges on the upper-left boundary of $K$
of color $x$ and $y$, respectively. 
Then there is no edge $Z$ of $K$ of color $z$ such that $z \prec x,y$.
\end{lemma}

\begin{proof}
By way of contradiction, let $Z$ be an edge of color $z$ such that $z \prec x,y$.
$\ell_Y$ the $\normaldown$line through $0_Y$.
By Corollary~\ref{C:ucovv1},  there exist edges $L_X, L_Y, Z_X, Z_Y$ such that 
$X \perspdn L_X \swing Z_X$,  $Y \perspdn L_Y \swing Z_Y$
and $Z \in \traj {Z_X} \ii \traj {Z_Y}$. 
This implies that $Z_X  \perspup Z_Y$, 
contradicting that $0_{Z_X}$ is meet-irreducible 
or that $Z_X  \perspdn Z_Y$, 
contradicting that $0_{Z_Y}$ is meet-irreducible.
\end{proof}

By Corollary~\ref{C:max}, the set of maximal elements of $\E P$ is the same 
as the set of colors of edges in the upper boundaries,
which set we can partition into the set of colors of edges in the upper-left 
and upper-right boundaries. 
No two distinct elements in the same subset have a common lower cover
by Lemma~\ref{L:disjoint}. 
This verifies the Partition Property.

\section{The Maximal Cover Property}\label{S:Maximal}

Let $x \in \E P$ be covered by a maximal element $y$ of $\E P$. 
By Corollary~\ref{C:ucovv1}, there is a covering $\SN 7$ such that
$\Col M = x$ and $\Col L = y$, using the notation of Figure~\ref{F:notation}.
Moreover, $\Col L$ and $\Col R$ are the only covers of $x$ in  $\E P$.
By Corollary~\ref{C:max}, 
there is an edge~$U$ in the upper-left boundary, 
so that $U \perspdn L$. 
Similarly, there is an edge~$V$ in the upper-right boundary, 
so that $V \perspdn R$.
So if the Maximal Cover Property fails, then  $\Col L = \Col R$.
This contradicts Corollary~\ref{C:ucovv}(i).

\section{The No Child Property}\label{S:Child}

By way of contradiction, let us assume that 
there are elements $a, b, c, d \in \E P$ 
with $b \neq c \prec a$, $d \prec b, c$ in~$\E P$,
where $a$ is maximal in $\E P$.
By Corollary~\ref{C:max}, the element~$a$ colors an edge~$A$ in the upper boundary of $K$,
say, in the upper-left boundary.

By Corollary~\ref{C:ucovv1}, we get a covering $\SN 7$,   
with middle edge~$B$, upper-left edge $L$ satisfying that $\Col B = b$ and $A \perspdn L$.
Similarly, we get another covering~$\SN 7$, with middle edge $C$ 
and upper-left edge $L'$ satisfying that $\Col B = b$ and $A \perspdn L'$. 
By Lemma~\ref{L:disj}, $\traj B$ and $\traj C$ are disjoint, 
contradicting the existence of the element $d \prec b,c \in \E P$,
see Corollary~\ref{C:ucovv}(ii).

\section{The Four-Crown Two-pendant Property}\label{S:Crown}

By way of contradiction, assume that the ordered set $\E R$ of Figure~\ref{F:notation}
is a cover preserving ordered subset of $\E P$,
where $a,b,c,d$ are maximal elements of $\E P$. 
By Corollary~\ref{C:max}, there are edges $A,B,C,D$ on the upper boundary of $K$, 
so that  $\col A = a$, $\col B = b$, $\col C=c$, $\col D = d$.
By left-right symmetry, we can assume that the edge $A$ 
is on the upper-left boundary of $K$. Since $p \prec a, b$ in  $\E P$,
it follows from Lemma~\ref{L:disjoint} that the edge $B$
cannot be on the upper-left boundary of $K$. 
So $B$ is on the upper-right boundary of $K$, and so is $D$.
Similarly, $C$ is on the upper-left boundary of $K$. 
So there are four cases, (i) $C$ is below $A$ and $B$ is below $D$;
(ii) $C$ is below $A$ and $D$ is below $B$; and so on. 
The first two are illustrated in Figure~\ref{F:CABDx}.

\begin{figure}[htb]
\centerline{\includegraphics[scale=1.2]{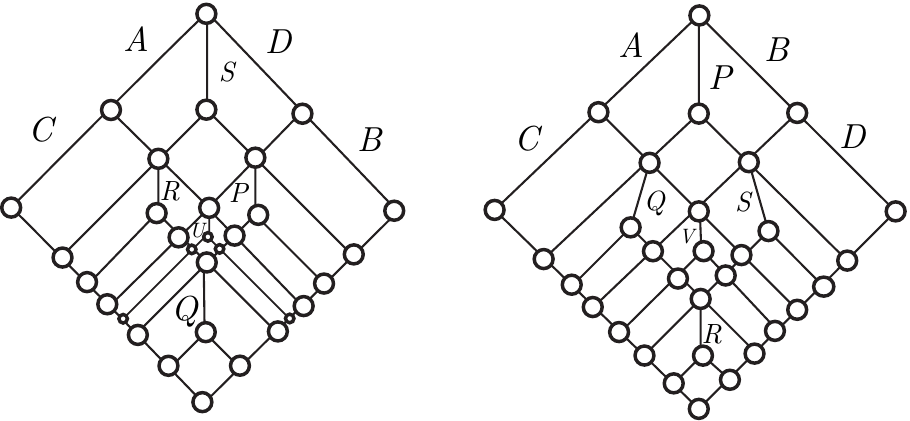}}
\caption{Illustrating the proof of The Four-Crown Two-pendant Property}
\label{F:CABDx}
\end{figure}

We consider the first case.
By Corollary~\ref{C:ucovv1} (see Figure~\ref{F:Corollary2}), 
there is a covering~$\SN 7$ with middle edge $P$ 
(as in the first diagram of Figure~\ref{F:CABDx})
so that $A$ and $B$ are down-perspective 
to the left-upper edge and the right-upper edge of the covering~$\SN 7$,
respectively. We define, similarly, the edge $Q$ for $C$ and $B$,
the edge $S$ for $A$ and~$D$, the edge $R$ for $C$ and $D$,
and the edge $U$ for $R$ and $P$.

The ordered set $\E R$ is a cover preserving subset of $\E P$,
so we get the covering~$\SN 7$ with middle edge $U$, 
upper-left edge $U_l$ and upper-right edge $U_r$;
the edge $U$ is collapsed by $\con P \mm \con R$.
Then $R \perspdn U_l$ and  $P \perspdn U_r$,
contradicting that  $\traj P$ and $\traj R$ do not meet in a $4$-cell
since $\ZL P$ and $\ZL R$ are both \normalupn\ and so parallel.
This conludes the proof of the Four-Crown Two-pendant Property
and of Cz\'edli's Theorem.

Of course, the diagrams in Figure~\ref{F:CABDx} are only illustrations.
The grid could be much larger, the edges $A, C$ and $B, D$ may not be adjacent, 
and there maybe lots of other elements in $K$. 
However, our argument only utilized what is true 
(for instance, that some edges are \normalupn) 
whatever the configuration.

The second case is similar, except that we get the edge $V$ and cannot get the edge $U$.
The third and fourth cases follow the same way.

\end{document}